\documentclass{svmult}

\usepackage{mathptmx}      
\usepackage{helvet}            
\usepackage{courier}           
\usepackage{graphicx}              
\usepackage[bottom]{footmisc}
\usepackage{color}
\usepackage{mathrsfs}%
\usepackage{amsmath}%
\usepackage{amsfonts}%
\usepackage{amssymb}%

\usepackage[abbrev]{amsrefs}
\usepackage{graphicx}

\newcommand{\NCpa}{\partial}

\newcommand{\NCN}{\mathbb{N}}
\newcommand{\NCR}{\mathbb{R}}

\newcommand{\NCcC}{\mathcal{C}}

        \definecolor{pink}{rgb}{1,0,1}
\smartqed
\makeindex             

\begin{document}

\title*{Eigenvalue Estimates on Bakry-\'Emery Manifolds}
\author{Nelia Charalambous, Zhiqin Lu, and Julie Rowlett}
\institute{Nelia Charalambous \at Department of Mathematics and Statistics, University of Cyprus, \email{nelia@ucy.ac.cy} \and Zhiqin Lu \at Department of Mathematics, University of California, Irvine, \email{zlu@uci.edu}
\and Julie Rowlett \at Institut f\"ur Analysis, Leibniz Universit\"at Hannover, \email{rowlett@math.uni-hannover.de}}

%
%
\maketitle

\abstract*{We demonstrate lower bounds for the eigenvalues of compact Bakry-\'Emery manifolds with and without boundary.  The lower bounds for the first eigenvalue rely on a generalised maximum principle which allows gradient estimates in the Riemannian setting to be directly applied to the Bakry-\'Emery setting.  Lower bounds for all eigenvalues are demonstrated using heat kernel estimates and a suitable Sobolev inequality.}

\abstract{We demonstrate lower bounds for the eigenvalues of compact Bakry-\'Emery manifolds with and without boundary.  The lower bounds for the first eigenvalue rely on a generalised maximum principle which allows gradient estimates in the Riemannian setting to be directly applied to the Bakry-\'Emery setting.  Lower bounds for all eigenvalues are demonstrated using heat kernel estimates and a suitable Sobolev inequality.}

\section{Introduction}
\label{intro}
Let $(M,g)$ be a Riemannian manifold and $\phi \in \NCcC^2 (M)$.  A \em Bakry-\'Emery \em manifold is a triple
$(M,g,\phi)$, where the measure on $M$ is the weighted measure $e^{-\phi} dV_g$.  The naturally associated Bakry-\'Emery Laplacian is
\begin{equation} \label{be-lap}
\Delta_\phi=\Delta - \nabla\phi\cdot\nabla,
\end{equation}   
where$$\Delta = \frac{1}{\sqrt{\det(g)}} \sum_{i,j} \NCpa_i g^{ij} \sqrt{\det(g)} \NCpa_j.$$

The operator can be extended as  a self-adjoint operator with respect
to the weighted measure $e^{-\phi} dV_g$.  It is also known as a ``drifting'' or ``drift'' Laplacian.  Bakry and \'Emery \cite{BE} observed that this generalised notion of Laplace operator has analogous properties to the standard Laplacian (which is none other than a Bakry-\'Emery Laplacian with $\phi \equiv 0$) and can be used to study a much larger class of diffusion equations and relations between energy and entropy.  For example, a Bakry-\'Emery Laplacian appears in the Ornstein-Uhlenbeck equation, and Bakry-\'Emery manifolds play a key role in the log-Sobolev inequality of Gross \cite{gr} and Federbush \cite{fed} as well as the hypercontractivity inequality of Nelson \cite{nel}.

The Bakry-\'Emery Laplacian has a canonically associated heat operator,
$$\NCpa_t-\Delta_\phi  .$$
The fundamental solution is known as the Bakry-\'Emery heat kernel.  The naturally associated curvature tensor for $(M, g, \phi)$ is the Bakry-\'Emery Ricci curvature  defined by
be\footnote{In the notation of \cite{lott}, this is the $\infty$
Bakry-\'Emery Ricci curvature.}
\[
{\rm Ric_\phi} ={\rm Ric}+{\rm Hess}(\phi).
\]
Above ${\rm Ric}$ and $\Delta$ are, respectively, the Ricci curvature and Laplacian with respect to the Riemannian metric $g$.

A collection of geometric results for Bakry-\'Emery manifolds is contained in \cite{ww}.  We are interested in the analysis of the Bakry-\'Emery Laplacian and associated heat kernel.  It turns out that some results can be extracted from the analysis of the Laplacian and heat kernel on an appropriately defined Riemannian manifold. In this article we are interested in obtaining new lower bounds for the eigenvalues of the Bakry-\'Emery Laplacian which are presented in Section  \ref{S2}.   Upper bounds are also known to hold for compact Bakry-\'Emery  manifolds.   We include here a brief summary of some recent results in this area in order to compare them to the lower bounds that we obtain. Although our survey is certainly not comprehensive, it gives a flavor of the type of estimates that one can show.

\section{Eigenvalue estimates} \label{S2}

\subsection{One-dimensional collapse}
We discovered in \cite{mrl} that the eigenvalues of a Bakry-\'Emery Laplacian on a compact $n$-dimensional manifold are the limit under one-dimensional collapse of Neumann eigenvalues for the classical Laplacian on a related $(n+1)$-dimensional manifold.

\begin{theorem}[L.-R.]\label{th-mrl}
Let $(M, g, \phi)$ be a compact Bakry-\'Emery manifold.  Let
$$M_{\epsilon} := \{ (x, y)\mid x \in M, \quad 0 \leq y \leq \epsilon e^{-\phi(x)} \} \subset M \times \NCR^+,$$
with $\phi \in \NCcC^2 (M)$ and $e^{-\phi} \in \NCcC (M \cup \NCpa M)$.  Let $\{\mu_k \}_{k=0} ^{\infty}$ be the eigenvalues of the Bakry-\'Emery Laplacian on $M$.  If $\NCpa M \neq \emptyset$, assume the Neumann boundary condition.  Let $\mu_{k}(\epsilon)$ be the Neumann eigenvalues of $M_\epsilon$ for $\tilde{\Delta} := \Delta + \NCpa_y ^2$.  Then 
$$ \mu_k(\epsilon)=\mu_k+O(\epsilon^2), \quad \forall k\geq 0.$$
\end{theorem}

\subsection{Maximum principle and gradient estimates}
One of the classical methods for obtaining eigenvalue estimates is via gradient estimates, which  was first used by Li-Yau~\cite{li-yau}. The papers~\cites{swyy,yau,yau-3,kroger,bakry-qian,zhong-yang,ac,yau-4} appear to be the most influential.   These estimates, which are often quite complicated and tricky,  are based on the following maximum principle.

Let $M$ be a compact Riemannian manifold. Let $u$ be a smooth function on $M$. Assume that
\[
H=\frac 12|\nabla u|^2 +F( u),
\]
where $F$ is a smooth function of one variable, and let $x_0$ be an interior point of $M$ at which $H$ reaches its maximum. Then at $x_0$
\[
0\geq|\nabla^2 u|^2+\nabla u\nabla(\Delta u)+{{\rm Ric}}(\nabla u,\nabla u)+F'( u)\Delta u+F''( u)|\nabla u|^2.
\]

The above inequality is useful for obtaining lower bounds on  the first eigenvalue of a Laplace or Schr\"odinger operator;  see ~\cite{s-yau}.  This together with the eigenvalue convergence under one-dimensional collapse in Theorem \ref{th-mrl} would indicate that similar estimates could be obtained for $M$ using $M_\epsilon$.  However there are two major problems with this naive approach:
\begin{enumerate}
\item $M_\epsilon$ need not be convex, even  if $M$ is.  As we know,
if $M$ is convex, the maximum of $H$ must be reached in the interior of $M$. In general, we don't have such a property for $M_\epsilon$.
\item The natural Ricci curvature attached to the problem is {${\rm Ric}_\phi$}, not the Ricci curvature of $M_\epsilon$, which is essentially {${\rm Ric}$.}
\end{enumerate}

Nonetheless, carefully estimating the eigenfunctions and their derivatives on $M_\epsilon$, we proved the following Maximum Principle for Bakry-\'Emery manifolds.
For a smooth function $u$ on $M_\epsilon$, we define the following function on $M$ 
\[
\psi(x):= u(x,0),\quad x\in M.
\]

\begin{theorem}[Maximum Principle (L.-R.)]
Assume that  $(x_0, 0)$ is the maximum point of $H$ on $M\times\{0\}\subset M_\epsilon$. Then
\[
o(1)\geq|\nabla^2\psi|^2+\nabla\psi\nabla(\Delta_\phi u)+{\rm Ric}_\phi(\nabla\psi,\nabla\psi)+F'(\psi){ \Delta_\phi} u+F''(\psi)|\nabla\psi|^2
\]
as $\epsilon\to 0$,
where $o(1)$ depends on certain weighted H\"older norm of $u$ (see ~\cite{mrl}*{Theorem 5} for details). In particular, if $u$ is an eigenfunction of unit $L^2$ norm, then $o(1)\to 0$ as $\epsilon\to 0$.
\end{theorem}

It therefore follows from this Maximum Principle that one may apply all the proofs of gradient estimates directly to Bakry-\'Emery geometry which we summarize as follows.  \\

{\bf Bakry-\'Emery Gradient Estimate Principle.} {\em There is a one-one correspondence between the gradient estimate on a Riemannian manifold and on a Bakry-\'Emery manifold. More precisely, the  eigenvalue estimate on the Bakry-\'Emery manidold $(M,g,\phi)$ is equivalent to that on the Riemannian manifold $(M_\epsilon, g+dy^2)$ for $\epsilon$ small enough.  }

\begin{remark}
It is known that the Bakry-\'Emery Laplacian  is unitarily equivalent  the Schr\"odinger operator
$$\Delta + \frac{1}{2} \Delta \phi + \frac{1}{4} | \nabla \phi|^2,$$
(see \cites{setti,ii-1,ii-2}). Using this observation, we are able to prove several eigenvalue inequalities in the Bakry-\'Emery setting virtually effortlessly, as long as the analogous results have been obtained in the Riemannian case. However, to obtain results involving gradient estimates, the equivalent estimates for the eigenfunctions are also required; these were demonstrated in ~\cite{mrl}.
\end{remark}

\begin{remark} Similar estimates can also be obtained by taking the warped product with the unit ball~\cite{CL2}, the advantage of that treatment being to avoid the boundary estimates.  However,  further work is necessary in that case to eliminate the extra eigenvalues which are created in that process.
\end{remark}

\subsection{Lower bounds for the first eigenvalue}
Using the Maximum Principle and Theorem~\ref{th-mrl}, we are able to  provide the Bakry-\'Emery version of the first eigenvalue estimates.
Throughout this subsection, let $(M, g, \phi)$ be a compact $n$-dimensional Bakry-\'Emery manifold either without boundary or with convex boundary, in which case we assume the Neumann boundary condition.
We first consider the case in which the (Bakry-\'Emery) Ricci curvature has a non-positive lower bound.

\begin{theorem} \label{thm1}
Assume ${\rm Ric}_\phi \geq -(n-1)k$ for some $k\geq 0$.  Then the first (positive) eigenvalue of the Bakry-\'Emery Laplacian $\mu_1$ satisfies
$$\mu_1 \geq \frac{\pi^2}{d^2}\exp(-c_n\sqrt{kd^2}),$$
where $d$ is the diameter of $M$ with respect to $g$, and $c_n$ is a constant depending only on $n$.
\end{theorem}

In the Riemannian case ($\phi\equiv 0$), the result is due to Yang~\cite{yang-2} following a similar idea of Zhong and Yang~\cite{zhong-yang}.

\begin{proof} Let $f$ be the first eigenfunction and assume without loss of generality 
\begin{align*}
&\max f=1;\\
&\min f=-\beta
\end{align*}
for some $0<\beta\leq 1$.
The following gradient estimate was demonstrated in ~\cite{yang-2}*{Lemma 2},
\begin{equation}\label{gra-5}
\frac{|\nabla f|}{\sqrt{1-f^2}}\leq\sqrt{\mu_1}+\frac 12\max(\sqrt{n-1},\sqrt 2)\sqrt{(n-1)k}\,\sqrt{1-f^2}
\end{equation}
for the Riemannian case. By our principle, the same estimate is true in the Bakry-\'Emery case. As a result, we have
\begin{equation}\label{16}
\mu_1\geq \frac{\pi^2}{16}\cdot\frac{\max({n-1},2) (n-1)k}{(\exp(1/2\max(\sqrt {n-1},\sqrt2)\sqrt{(n-1)kd^2})-1)^2}
\end{equation}
which is obtained by integrating ~\eqref{gra-5} over the geodesic connecting the maximum and minimum points of $f$.

Define the {\em normalized} eigenfunction
\[
\varphi:=\frac{f-(1-\beta)/2}{(1+\beta)/2}
\]
so that $\max \varphi=1$ and $\min\varphi=-1$. Similarly, the following gradient estimate  in ~\cite{yang-2}*{Lemma 5},
\[
|\nabla\varphi|^2\leq \mu_1+(n-1)k+\mu_1\xi(\varphi),
\]
is also true in the Bakry-\'Emery case,
where $ \mu_1+(n-1)k+\mu_1\xi(\varphi)$ satisfies an ordinary differential equation in~\cite{yang-2}*{eq. (40)}.   Consequently, we have
\begin{equation}\label{grad-1}
\mu_1\geq\frac{\pi^2}{d^2}\cdot\frac{1}{1+(n-1)k/\mu_1}.
\end{equation}
Combining the above inequality with~\eqref{16} proves the theorem;  for further details we refer to ~\cite{yang-2}.
\end{proof}

 When the (Bakry-\'Emery) Ricci curvature has a positive lower bound, we obtain  a result of Futaki-Sano~\cite{futaki-sano} by our maximum principle and  the corresponding Riemannian case proven by Ling~\cite{ling}.

\begin{theorem}  Assume that ${\rm Ric}_\phi\geq (n-1) k$ for some positive constant $k>0$. Then the first (positive)  eigenvalue of the Bakry-\'Emery Laplacian satisfies
\begin{align*}
&\mu_1\geq \frac{\pi^2}{d^2}+\frac 38(n-1)k, & {\rm for } \,\,\, n=2;\\
&\mu_1\geq \frac{\pi^2}{d^2}+\frac{31}{100}(n-1)k, & {\rm for } \,\,\, n>2,
\end{align*}
where $d$ is the diameter of the manifold.
\end{theorem}

A slightly stronger estimate was shown by Andrews and Ni \cite{a-ni} for convex domains in a Bakry-\'Emery manifold.  

\begin{theorem}[Andrews-Ni]  Assume that for the Bakry-\'Emery manifold $(M, g, \phi)$ the associated curvature ${\rm Ric}_\phi\geq (n-1)k >0$.  Then for any convex domain $\Omega$, the first positive eigenvalue of the Bakry-\'Emery Laplacian with Neumann boundary condition satisfies 
 $$\mu_1 \geq \frac{\pi^2}{d^2} + \frac{(n-1)k}{2},$$
 where $d$ is the diameter of $\Omega$.  
\end{theorem}

\subsection{Lower bounds of higher eigenvalues}

We will also demonstrate a lower bound for \em all \em the eigenvalues which holds whenever the manifold satisfies an appropriate Sobolev inequality \eqref{Sob} as in Section \ref{S3} below.

\begin{theorem} \label{thm4}
Let $M$ be a compact manifold without boundary, on which the Sobolev inequality \eqref{Sob} holds. Then the $k$th eigenvalue of the Bakry-\'Emery Laplacian satisfies the lower bound
\[
\lambda_k^{\frac{\nu}{2}}  \geq c(\nu)\, k  \, V_{\phi}^{-1} \, C_1^{\frac{\nu}{2}}
\]
where $V_{\phi}$ is the weighted volume of $M$ and $c(\nu) >0$ is a uniform constant that only depends on $\nu$ and $C_1$ is as in Lemma \ref{L1}.  The same inequality holds when $\partial M \neq \emptyset$ for the positive Neumann eigenvalues.
\end{theorem}

\begin{remark}
We note that in the case $\partial M \neq \emptyset$, if the  Sobolev inequality \eqref{Sob} holds for all $u|_{\partial M} =0$ with constant $C_2$, then the $k$th Dirichlet eigenvalue of the Bakry-\'Emery Laplacian satisfies the same inequality with constant $C_2$.
\end{remark}

Assuming the Sobolev inequalities on $M_\epsilon$, then the result follows from Cheng-Li \cite{CheLi} and Theorem~\ref{th-mrl}.  This however, would entail a uniform Sobolev constant for all the  $M_\epsilon$. Instead, we shall prove this result using the   Bakry-\'Emery heat kernel estimates which we will present in the following sections.

\subsection{Upper bounds}
In this section we provide some recent upper bound estimates for eigenvalues of Bakry-\'Emery manifolds. We first recall that in the case of a Riemannian manifold, Cheng proved the following (see \cite{s-yau}*{Theorem III.2})

  \begin{theorem}[Cheng]
  Let $M$ be a compact Riemannian manifold without boundary or with Neumann boundary condition. Let $d$ be the diameter of $M$. Then for $j\geq 1$,
  \begin{enumerate}
  \item If ${ \rm Ric}\geq 0$, then $\mu_j\leq 8n(n+4)j^2/d^2$;
  \item If ${\rm Ric}\geq n-1$, then $\mu_j\leq 4nj^2/d^2$;
  \item If ${\rm Ric}\geq -(n-1)k$ for $k\geq 0$, then $\mu_j\leq \frac 14k+8n(n+4)j^2/d^2$
  \end{enumerate}
  \end{theorem}

Using ~\cite{CL2},  the above inequalities are true under slightly stronger assumptions in the Bakry-\'Emery case.

  \begin{theorem}\label{new-cheng}
  Let $M$ be a compact Bakry-\'Emery manifold without boundary or with Neumann boundary condition. Let $d$ be the diameter of $M$.  Let $\epsilon>0$.
 If 
 $${\rm Ric}_\phi-\epsilon \nabla\phi\otimes\nabla\phi\geq -(n-1)k, \quad \textrm{for} \quad k\geq 0,$$ 
 then 
 $$\mu_j\leq C(n,\epsilon)(k+j^2/d^2), \quad \forall j \in \NCN,$$ 
 where $C(n,\epsilon)$ is a constant depending on $n$ and $\epsilon$.
 \end{theorem}

Using this result we are above to prove the following which is essentially due to~\cite{setti,ii-1,ii-2}.

\begin{theorem} Assume that ${\rm Ric}_\phi-\epsilon\nabla\phi\otimes\nabla\phi\geq 0$. Then we have
\[
\mu_j\leq C(n,\epsilon)\mu_1.
\]
\end{theorem}

\begin{proof} By Theorem~\ref{thm1}, we have
\[
\frac{\pi^2}{d^2}\leq\mu_1.
\]
The result therefore follows from Theorem~\ref{new-cheng}.

\end{proof}

Recently, Funano and Shioya proved \cite{funshi} the following stronger and somewhat surprising  result.  

\begin{theorem}[Funano-Shioya] Let $(M, g, \phi)$ be a compact Bakry-\'Emery manifold with non-negative Bakry-\'Emery Ricci curvature.  Then there exists a positive constant $C_j$ which depends only on $j$ (and not even on the dimension!) and in particular is independent of $(M, g, \phi)$ such that
$$\mu_j \leq C_j \mu_1.$$
Moreover, this result also holds if the $\partial M \neq \emptyset$ is $C^2$ under the Neumann boundary condition.
\end{theorem}

Using an example, they showed that the non-negativity of curvature is a necessary condition.  The proof relies on a geometric theory of concentration of metric measure spaces due to Gromov \cite{gr07}.\\

Hassannezhad demonstrated upper bounds for the eigenvalues without curvature assumptions \cite{hass}.

\begin{theorem}[Hassannezhad] There exist constants $A_n$ and $B_n$ depending only on $n$ such that for every $n$-dimensional compact Bakry-\'Emery manifold $(M, g, \phi)$ with $|\nabla \phi| \leq \sigma$ for some $\sigma \geq 0$, and for every $j \in \NCN$ we have
$$\mu_j \leq A_n \max\{\sigma^2, 1\} \left( \frac{ V_\phi ([g])}{V(M,g)} \right)^{2/n} + B_n \left( \frac{j}{ V(M,g)} \right)^{2/n}.$$
Above $V(M,g)$ denotes the volume of $M$ with respect to $g$, and $V([g])$ denotes the min-conformal volume,
$$V([g]) = \inf \{ V(M, g_0), \textrm{ such that $g_0 \in [g]$, and {\rm Ric}$(g_0) \geq - (n-1)$}\}.$$
\end{theorem}

This theorem was proven by first demonstrating an analogous estimate for the Schr\"odinger operator
$$\Delta + \frac{1}{2} \Delta \phi + \frac{1}{4} | \nabla \phi|^2,$$
which is unitarily equivalent to the Bakry-\'Emery Laplace operator; see \cite{setti}*{p. 28}.

The proof of the following theorem is based on constructing a family of test function supported on a suitable family of balls and is known as a Buser type upper bound, since this idea goes back to Buser \cite{bus}, and has also been used by Cheng \cite{cheng} as well as Li and Yau \cite{li-yau12}.

\begin{theorem}[Hassannezhad] There are positive constants $A_n$ and $B_n$ which depend only on the dimension $n$ such that for every compact Bakry-\'Emery manifold $(M, g, \phi)$ with ${\rm Ric}_\phi \geq - k^2 (n-1)$ and $| \nabla \phi | \leq \sigma$, for some constants $k $, $\sigma \geq 0$, such that for every $j \in \NCN$ we have
$$\mu_j \leq A_n {\rm max} \{\sigma^2, 1\} k^2 + B_n \left( \frac{j}{V_\phi(M)} \right)^{2/n},$$
where
$$V_\phi(M) := \int_M e^{- \phi} dV_g$$
is the weighted volume of $M$.
\end{theorem}

\section{Sobolev inequalities} \label{S3}
A classical way to obtain lower bounds on the eigenvalues of the Laplacian on a compact set is via the trace of the heat kernel as in \cite{Dod2}. Cheng and Li demonstrated in \cite{CheLi} that one can also find such lower bounds with respect to the Sobolev constant since a Sobolev inequality always holds in the compact case. Their method also ultimately relies on demonstrating upper bounds for the heat trace as in \cite{Dod2}.

\begin{definition} \label{SI}
We say that the  Bakry-\'Emery manifold $(M^n,g,\phi)$ satisfies the property  \eqref{Sob}, if there exist constants $ \nu= \nu(n)>2$, $ \alpha =\alpha (n)$, and $C_o$ depending only on $M$ such that for all  $ u\in \mathcal H^1 (M)$
\begin{equation} \label{Sob}
\left(\int_{M} | u|^{\frac{2\nu}{\nu-2}}\, e^{-\phi} \right)^{\frac{\nu-2}{\nu}} \leq C_o  \; V_\phi^{-\frac{2}{\nu}}\; \int_{M}(\,|\nabla  u|^2 + \alpha  | u|^2\,)\, e^{-\phi} \tag{$S$}
\end{equation}
where $V_\phi$ denotes the weighted volume of $M$.
\end{definition}

A global Sobolev inequality as above is known to hold on compact Riemannian manifolds. In a recent article, the first two authors found sufficient conditions for a local Sobolev inequality to hold on a weighted manifold \cite{CL1}. The local Sobolev inequality points to the geometric features upon which $C_o$ would depend in the case of a weighted manifold. In particular, the authors showed that a volume form comparison assumption is sufficient to ensure a local Sobolev inequality.  For any point $x\in M$ we denote the Riemannian volume form in geodesic coordinates at $x$ by
\[
dv(\exp_x(r\xi))=J(x,r,\xi) \, dr \, d\xi
\]
for $r>0$ and $\xi$ any unit tangent vector at $x$. Then the $\phi$-volume form in geodesic coordinates is given by
\[
J_\phi(x,r,\xi) = e^{-\phi} J(x,r,\xi).
\]
If $y=\exp_x(r\xi)$ is a point that does not belong to the cut-locus of $x$, then
\[
\Delta r(x,y) = \frac{J'(x,r,\xi)}{J(x,r,\xi)} \ \ \ \text{and} \ \ \  \Delta_\phi r(x,y) = \frac{J_\phi'(x,r,\xi)}{J_\phi(x,r,\xi)}
\]
where $r(x,y)=d(x,y)$, and the derivatives are taken in the radial direction. The first equality gives Bishop's volume comparison theorem  under the assumption of a uniform  Laplacian upper bound. On weighted manifolds, the second equality provides us with weighted volume comparison results whenever we have a uniform Bakry-\'Emery Laplacian upper bound.

\begin{definition} \label{VF}
We say that the Bakry \'Emery manifold $(M^n,g,\phi)$ satisfies the property \eqref{VR}, if there exists a positive and nondecreasing function $A(R)$ and a uniform constant $a$ (independent of $R$) such that for all $x\in B_{x_o}(R)$ and $0<r_1<r_2<R$
\begin{equation} \label{VR}
\frac{J_\phi(x,r_2,\xi)}{J_\phi(x,r_1,\xi)} \leq \left(  \frac{r_2}{r_1}\right)^a \, e^{A(R)}. \tag{$V_R$}
\end{equation}
The above inequality is assumed for all points  $\exp_x(r_i\xi)$ that do not belong to the cut locus of $x$.
\end{definition}

We denote by $B_x(r)$ the geodesic ball of radius $r$ at $x$ and by $V_\phi(x,r)$ its weighted volume.  The following result was proven in \cite{CL1}.
\begin{lemma}
Let   $(M^n,g,\phi)$    be a Bakry-\'Emery manifold that satisfies the property \eqref{VR} for all $x\in B_{x_o}(R)$.
Then for any $x\in B_{x_o}(R)$, $0<r<R$ and  $ u\in \mathcal C^{\infty}_0 (B_x(r))$ there exist constants $ \nu= \nu(n)>2$, $C_1(n,a)$ and $C_2(n)$ such that
\begin{equation} \label{Sob2}
\left(\int_{B_x(r)} | u|^{\frac{2\nu}{\nu-2}}\,  e^{-\phi} \right)^{\frac{\nu-2}{\nu}} \leq C_1 \frac{e^{C_2 A(R) }\, r^2 }{V_\phi(x,r)^{\frac{2}{\nu}}} \int_{B_x(r)}(\,|\nabla  u|^2 + r^{-2}| u|^2\,)\, e^{-\phi}.
\end{equation}
\end{lemma}

Previously, similar local Sobolev inequalities were proven in the case of a uniform upper bound on $\Delta_\phi r$. Assumption \eqref{VR}, however, only requires that the integral of $\Delta_\phi r$ on a geodesic ball be bounded and is thus more general. We refer the interested reader to \cite{CL1} for specific conditions on $\text{Ric}_\phi $ and $\phi$ that would guarantee such a uniform upper bound. On a compact manifold they all certainly hold. An interesting question we intend to investigate in future work is to determine the optimal $C_o$ and $\alpha$ of \eqref{Sob}.  We would also like to remark that in the case $\phi\equiv 0$ one can use the existence of a local Sobolev inequality \eqref{Sob2} to find lower bounds for the Neumann eigenvalues of the Laplacian over a geodesic ball. In \cite{Char} such lower estimates were also obtained for the Bochner Laplacian on forms. It would also be interesting to consider the analogous problem on weighted manifolds.

The Sobolev inequality \eqref{Sob} allows us to prove an $L^2$ gradient estimate which together with the heat kernel estimates will be sufficient to prove the eigenvalue lower bounds.

\begin{lemma} \label{L1}
Suppose that \eqref{Sob} holds on $M$. Then for all $ u\in H^1(M)$ that satisfy $\int_M  u =0$
\[
\int_M |\nabla  u|^2\,  e^{-\phi} \geq C_1 \left(\int_M  u^2 \, e^{-\phi} \right)^{\frac{2+\nu}{\nu}} \left(\int_M | u| \, e^{-\phi}  \right)^{-\frac{4}{\nu}}
\]
for a uniform constant $C_1=\frac{\lambda_1}{C_o\,(\lambda_1+\alpha)} \; V_\phi^{\frac{2}{\nu}}$, where $\lambda_1$ is the first nonzero eigenvalue of $M$.
\end{lemma}
\begin{proof}
The Sobolev inequality \eqref{Sob} implies
\begin{equation} \label{L1e1}
\int_{M} |\nabla  u|^2  \geq C_o^{-1}\, V_\phi^{\frac{2}{\nu}}\; \left( \int_{M} | u|^{\frac{2\nu}{\nu-2}}\, e^{-\phi} \right)^{\frac{\nu-2}{\nu}} - \alpha  \int_{M}  | u|^2 \,  e^{-\phi}.
\end{equation}
Moreover, whenever $\int_M  u =0$, the definition of $\lambda_1$ gives
\[
 \int_M    u^2\,  e^{-\phi}  \geq \frac{1}{\lambda_1}  \int_M |\nabla  u|^2\,  e^{-\phi}.
\]
Substituting the above inequality to the right side of \eqref{L1e1} and solving for $\int_M |\nabla  u|^2$ we get
\begin{equation} \label{L1e2}
  \int_M |\nabla  u|^2\,  e^{-\phi} \geq \frac{\lambda_1}{C_o\,(\lambda_1+\alpha)}  \; V_\phi^{\frac{2}{\nu}}\; \left( \int_{M} | u|^{\frac{2\nu}{\nu-2}}\,  e^{-\phi} \right)^{\frac{\nu-2}{\nu}}.
\end{equation}
By writing $ u^2 = | u|^{4/(\nu+2)}| u|^{2\nu/(\nu+2)}$ and applying the H\"older inequality with $p=(\nu+2)/4$ and $q=p/(p-1)=(\nu+2)/(\nu-2)$ we get the estimate
\[
 \left(\int_M  u^2 \,  e^{-\phi} \right)^{\frac{2+\nu}{\nu}} \leq \left(\int_M  | u|\,  e^{-\phi}  \right)^{\frac{4}{\nu}} \left(\int_M | u|^{\frac{2\nu}{\nu-2}} \,  e^{-\phi}  \right)^{\frac{\nu-2}{\nu}}.
\]
The lemma follows by solving the above inequality for the second term in the right side and substituting into \eqref{L1e2}.
\end{proof}

\section{Heat kernel estimates and the non-compact case}
We let $H_\phi(x,y,t)$ denote the heat kernel of $\Delta_\phi$ corresponding to the Friedrichs extension. This is certainly unique when $M$ is compact, and on a noncompact manifold it is the smallest positive heat kernel among all other heat kernels that correspond to heat semi-groups of   self-adjoint extensions of $\Delta_\phi$. Both in the Riemannian and in the weighted case, heat kernel estimates are closely related to eigenvalue estimates. When $\phi \equiv 0$ Li and Yau in ~\cite{li-yau} prove upper estimates for the heat kernel of Schr\"odinger operators whenever the Ricci curvature of the manifold is bounded below. One of the key elements in their proof is the Bochner formula and the Cauchy inequality $|\nabla^2u|^2\geq (\Delta u)^2/n$.

In the case $\phi \not\equiv 0$, it was shown by Bakry and \'Emery in ~\cite{BE} that the analogous Bochner formula can be obtained if one takes as the curvature tensor $\textup{Ric}_\phi$, and it is given by
\begin{equation} \label{fBochner}
\Delta_\phi|\nabla u|^2=2 |\nabla^2 u|^2+2\langle\nabla u, \nabla\Delta_\phi u\rangle + 2 \textup{Ric}_\phi(\nabla u, \nabla u).
\end{equation}
Observe that when $\phi\equiv 0$, \eqref{fBochner} becomes the  Bochner formula in the Riemannian case.  The term $\nabla^2 u$ that appears above is the usual Hessian of $u$; we do not have a notion of $\phi$-Hessian. In other words there is no analogous relationship between the Hessian of $u$ and the $\Delta_\phi u$ as in  $|\nabla^2u|^2\geq (\Delta u)^2/n$. As a result, there is more subtlety in obtaining gradient and heat kernel estimates estimate in the $\phi \not\equiv 0$ case. Such estimates can be obtained under various assumptions on the curvature of the manifold. We will present a few of these options.

Bakry and \'Emery also demonstrated that the relevant Ricci tensor for obtaining a gradient estimate is the $q$-Bakry-\'Emery Ricci tensor~\cite{BE} which is defined as
\[
\textup{Ric}_\phi^q=\textup{Ric}+ \nabla^2 \phi-\frac 1q\nabla \phi\otimes\nabla \phi=\textup{Ric}_\phi-\frac 1q\nabla \phi\otimes\nabla \phi
\]
where $q$ is a positive number. By generalizing the Li-Yau  method, Qian was able to prove a Harnack inequality and heat kernel estimates for the Bakry-\'Emery Laplacian whenever $\textup{Ric}_\phi^q\geq 0$ in ~\cite{Qi}. In \cite{CL2} the first two authors found Gaussian estimates for the Bakry-\'Emery heat kernel whenever  $\textup{Ric}_\phi^q$ is bounded below. This was done by associating to the weighted manifold a family of warped product spaces $\tilde M_\epsilon$ and showing that the geometric analysis results on  $M$ are closely related to those  on $\tilde M_\epsilon$. In particular, the heat kernel estimates on $\tilde M_\epsilon$ implied the Bakry-\'Emery  heat kernel estimates on the weighted manifold.
\begin{theorem}[C. -L.] \label{Tb1}
Let  $(M^n,g,\phi)$    be a Bakry-\'Emery manifold  such that for some positive integer $q$,
\[
\textup{Ric}^q_\phi\geq -K
\]
on $B_{x_o}(4R+4)\subset M$. Then for any  $x,y\in B_{x_o}(R/4)$, $t<R^2/4$ and $\delta\in(0,1)$
\begin{equation*}
\begin{split}
 C_6  \, &    V_\phi^{-1/2}(x,\sqrt{t})\,  V_\phi^{-1/2}(y,\sqrt{t}) \cdot \exp{ [\,-C_7\, \frac{ {d}^2(x,y)}{t} - C_8 \,K \,t \,]}\\
& \leq H_\phi(x,y,t) \\
& \leq   \;C_3    \;   V_\phi^{-1/2}(x,\sqrt{t}) \,V_\phi^{-1/2}(y,\sqrt{t})
 \cdot  \exp [\,-\lambda_{1,\phi}(M)\,t  -\frac{d^2(x,y)}{C_4 \,t} + C_5\,\sqrt{K\,t} \,]
\end{split}
\end{equation*}
for some positive constants $C_3,  C_4, C_6$ and $C_7$ that only depend on  $\delta$ and $n+q$ and positive constants $C_5, C_8$ that only depend on $n+q,$ and where $\lambda_{1,\phi}(M)$ is the infimum of the weighted Rayleigh quotient on $M$.

Whenever $\textup{Ric}^q_\phi\geq -K$ on $M$, then the same bound also holds for all $x,y \in M$ and $t>0$.
\end{theorem}

The proof of the above theorem illustrated the strong geometric connection between $M$ and the warped product spaces $\tilde M_\epsilon$ and the fact that the Bakry-\'Emery Laplacian and the $q$-Bakry-\'Emery Ricci tensor are projections (in some sense) of the  Laplacian and Ricci tensor of a higher dimensional space. As it was also remarked in \cite{CL2}, one could not get the above estimate by only assuming $\textup{Ric}_\phi$ bounded below and $\phi$ of linear growth at a point (which would be enough for gradient estimates as in \cite{MuW2}). Instead, the Bakry-\'Emery  heat kernel estimate requires  an assumption on the {\it uniform} linear growth of $\phi$, which is almost equivalent to assuming that the gradient of $\phi$ is bounded.

\subsection{The essential spectrum}
On noncompact manifolds, a more interesting part of the spectrum is the essential spectrum of the manifold. In general, the $L^2$  spectrum of $\Delta_\phi$,  denoted $\sigma(\Delta_\phi)$,  consists of all points $\lambda\in \mathbb{C}$ for which $\Delta_\phi-\lambda I$ fails to be invertible on $L^2$. Since $\Delta_\phi$ is nonnegative definite on $L^2$, $\sigma(\Delta_\phi)$  is contained in $[0,\infty)$.   The  essential spectrum of  $\Delta_\phi$ on $L^2$, $\sigma_\textup{ess}(\Delta_\phi)$, consists of the cluster points in the spectrum and of isolated eigenvalues of  infinite multiplicity.  One is usually interested in finding sufficient conditions on the manifold so that $\sigma_\textup{ess}(\Delta)=[0,\infty)$.

In the case $\phi=0$, it  was extremely difficult to directly compute the $L^2$ spectrum using the classical Weyl's criterion without assuming very strong decay conditions on the curvature of the manifold. By generalizing the Weyl's criterion, the first two authors were able to show that a sufficient condition for such a result is that the Ricci curvature of the manifold is asymptotically nonnegative \cite{CL3}. The authors were also able to show that in the case of a weighted manifold $\sigma_\textup{ess}(\Delta_\phi)=[0,\infty)$ whenever the $q$-Bakry-\'Emery Ricci tensor is asymptotically nonnegative \cite{CL2}. In the same article, they also showed an $L^p$ independence result for the essential spectrum of $\Delta_\phi$ on $L^p$ whenever $\textup{Ric}_\phi^q$  is bounded below, and the weighted volume of the manifold grows uniformly subexponentially \cite{CL2}. As in the classical case, the latter result is a consequence of the gaussian estimates for the heat kernel \cite{sturm}. It would be interesting to find a weighted space analog to the hyperbolic space on which the $L^p$  essential spectrum of $\Delta_\phi$ depends on $p$, for which the underlying Riemannian manifold is not hyperbolic (see \cite{DST} for the classical case).

As mentioned previously, in the noncompact case, $\textup{Ric}_\phi^q$ bounded below is not equivalent to $\textup{Ric}_\phi$ bounded below, since $\phi$ and its gradient are not necessarily bounded. As a result, if one would like to assume instead $\textup{Ric}_\phi\geq -K$, then some control on $\phi$ is required. Apart from the Bochner formula, the other main ingredient for obtaining the heat kernel estimates by the Li-Yau method in the Riemannian case is the Laplacian comparison theorem. In particular, the fact that whenever the Ricci curvature of the manifold is bounded below, then there exist uniform constants $a, b$ such that
\[
\Delta  r(x,y) \leq  \frac{a}{r(x,y)} + b
\]
where $r(x,y)=d(x,y))$. In the case of weighted manifolds however, ${\rm Ric}_\phi (x)\geq-  K $ on a ball around $x_o$ does not imply a uniform Bakry-\'Emery Laplacian estimate $\Delta_\phi r(x,y) \leq C\frac{1}{r(x,y)} + b$  without strong restrictions on $\phi$. However, using the techinique of Saloff-Coste as in  \cite{SCbk}, one can use the local Sobolev inequality \eqref{Sob2} to prove a mean value inequality for $\phi$-subharmonic functions as well as a mean value inequality for solutions to the Bakry-\'Emery  heat equation. In \cite{CL2} these were used to prove a Gaussian estimate for the Bakry-\'Emery  heat kernel.
\begin{theorem}[C. -L.] \label{T2}
Let  $(M^n,g,\phi)$    be a Bakry-\'Emery manifold  that satisfies the property \eqref{VR} for all $x\in B_{x_o}(R)$.
Let $H_\phi(x,y,t)$ denote the minimal Bakry-\'Emery  heat kernel  defined on $M\times M\times (0,\infty)$
Then for any  $\epsilon >0$ there exist constants $c_1(n,\epsilon), c_2(n)$ such that
\[
H_\phi(x,y,t) \leq  c_1  \, V_\phi^{-1/2}(x,\sqrt{t}) \; V_\phi^{-1/2}(y,\sqrt{t}) \; \exp[-\lambda_{1,\phi}(M) t  - \frac{d^2(x,y)}{4(1+\epsilon)\, t} + c_2\, A(R)]
\]
for any $x,y \in B_{x_o}(R/2)$ and $0<t<R^2/4$.
\end{theorem}
For a comprehensive review of heat kernel bounds on noncompact weighted manifolds we refer the interested reader to the extensive summary of results due to Grigor\'{ }yan in \cite{Grig2} as well as the book \cite{Grig1},  where one can find various equivalence relationships between gaussian heat kernel bounds, Poincar\'e inequalities and volume doubling, the relative Faber Krahn inequality as well as the Harnack inequality.

\subsection{Proof of Theorem \ref{thm4}}
\begin{proof}

Let $H_\phi(x,y,t)$ be the heat kernel of the Bakry-\'Emery Laplacian on $M$ (Neumann kernel in case of boundary). Using the eigenvalues of $\Delta_\phi$, the Bakry-\'Emery heat kernel has the following expression
\[
H_\phi(x,y,t)=\sum_{i=0}^{\infty} e^{-\lambda_i t} \phi_i(x) \phi_i(y)
\]
where $\phi_i(x)$ is the eigenfunction corresponding to $\lambda_i$ chosen such that $\{\phi_i\}_{i=0}^\infty$ are orthonormal in the weighted $L^2$ norm. Note that $\lambda_0=0$ and $\phi_0= V_\phi^{-1/2}$, and as a result $\int_M \phi_i(y) \, e^{-\phi(y)} =0$ for all $i\geq 1$. To obtain lower estimates for the eigenvalues, the idea is to find a uniform upper bound for the trace of the Bakry-\'Emery  heat kernel, namely $ H_\phi(x,x,t)$. In fact, it will sufficient to find a uniform upper bound for
\[
G(x,y,t) = H_\phi(x,y,t)-\frac{1}{V_\phi} =\sum_{i=1}^{\infty} e^{-\lambda_i t} \phi_i(x) \phi_i(y)
\]
from the above remark. We observe that $G$ satisfies
\[
\int_M G(x,y,t)  \, e^{-\phi(y)} =0
\]
and the semigroup property
\[
G(x,z,t+s) = \int_M G(x,y,t) \, G(y,z,s) \; e^{-\phi(z)}
\]
for all $x,z \in M$ and $t, s \in [0,\infty)$. The definition of $G$ together with the properties $\int_M H_\phi(x,y,t)\, e^{-\phi(y)}=1$ and $H_\phi(x,y,t)\geq 0$ imply
\begin{equation} \label{thm2e1}
\int_M |G(x,y,t)|\, e^{-\phi(y)} \leq \int_M |H_\phi(x,y,t)|\, e^{-\phi(y)} +1 = 2.
\end{equation}
From the semigroup property
\[
G(x,x,t) = \int_M G(x,y,t/2) \, G(x,y,t/2) \; e^{-\phi(y)}.
\]
Differentiating both sides with respect to $t$ we obtain
\begin{equation*}
\begin{split}
G'(x,x,t) &=  \int_M G'(x,y,t/2) \, G(x,y,t/2) \; e^{-\phi(y)}\\
&=   \int_M \Delta_{\phi,y}G(x,y,t/2) \, G(x,y,t/2) \; e^{-\phi(y)}
\end{split}
\end{equation*}
since $G$ also solves the heat equation. Integration by parts now gives
\begin{equation*}
\begin{split}
-G'(x,x,t)&=   \int_M |\nabla_{y}G(x,y,t/2)|^2 \; e^{-\phi(y)} \\
& \geq 2^{-\frac{4}{\nu}} C_1   \left(\int_M |G(x,y,t/2)|^2 \; e^{-\phi(y)} \right)^{\frac{2+\nu}{\nu}}
\end{split}
\end{equation*}
by Lemma \ref{L1} and equation \eqref{thm2e1}. By the semigroup property of $G$,
\[
-G'(x,x,t)\, (G(x,x,t))^{-(2+\nu)/\nu} \geq   2^{-4/\nu}\,   C_1
\]
Integrating both sides with respect to $t$ gives
\[
\frac{\nu}{2}\,(G(x,x,t))^{-2/\nu} \geq  2^{-4/\nu}\,  C_1   \, t
\]
since $G$ tends to $+\infty$ as $t$ goes to zero. Finally we have the upper estimate
\[
G(x,x,t)\leq  4\,\left( \frac{\nu}{2\, C_1} \right)^{\frac{\nu}{2}} \; t^{-\frac{\nu}{2}}.
\]
If we now integrate both sides with respect to $x$, the eigenvalue expansion for $G$ gives
\[
\sum_{i=1}^{\infty} e^{-\lambda_i t} \leq  4\,\left( \frac{\nu}{2\, C_1} \right)^{\frac{\nu}{2}} \; t^{-\frac{\nu}{2}}\; V_\phi.
\]
Setting $t=1/\lambda_k$ we get
\[
4\,\left( \frac{\nu\, \lambda_k}{2\, C_1} \right)^{\frac{\nu}{2}} \;   V_\phi \geq \sum_{i=1}^{\infty} e^{-\lambda_i /\lambda_k} \geq \frac{k}{e}
\]
since $\lambda_i /\lambda_k\leq 1$ for $i\leq 1$. The lower estimate follows.

In case $\partial M \neq \emptyset$, the proof is essentially identical and is left to the reader, one would just have to consider the Neumann or Dirichlet heat kernel, and the appropriate Sobolev inequality.
\end{proof}

\begin{acknowledgement}
The second author is partially supported by the   NSF grant DMS-12-06748.  The third author gratefully acknowledges the support of the Leibniz Universit\"at Hannover and the Australian National University.
\end{acknowledgement}
%


%
%


\end{document}